\documentclass[11pt]{amsart}   % includes 11pt font size (default 10pt)

\usepackage{verbatim}
\usepackage{amssymb,amsthm,amsmath}
\usepackage{color}
\definecolor{darkblue}{rgb}{0, 0, .4}
\definecolor{grey}{rgb}{.7, .7, .7}
\usepackage[breaklinks]{hyperref}
\hypersetup{
    colorlinks=true,
    linkcolor=darkblue,
    anchorcolor=darkblue,
    citecolor=darkblue,
    pagecolor=darkblue,
    urlcolor=darkblue,
    pdftitle={},
    pdfauthor={}
}

\newtheorem{theorem}{Theorem}[section]
\newtheorem{lemma}[theorem]{Lemma}

\theoremstyle{definition}
\newtheorem{definition}[theorem]{Definition}
\newtheorem{example}[theorem]{Example}

\theoremstyle{remark}
\newtheorem{remark}[theorem]{Remark}

\numberwithin{equation}{section}

\theoremstyle{theorem}
\newtheorem{corollary}[theorem]{Corollary}

\newcommand{\s}[0]{\sigma}

\usepackage[all, dvips, knot]{xy}
\usepackage[all, knot, dvips, color]{xypic}
\xyoption{arc}

\newcommand{\w}{\mathsf{w}}

\begin{document}

\title{An explicit derivation of the M\"obius function for Bruhat order}

\begin{abstract}
We give an explicit nonrecursive complete matching for the Hasse diagram of
the strong Bruhat order of any interval in any Coxeter group.  This yields a
new derivation of the M\"obius function, recovering a classical result due to
Verma.
\end{abstract}

\author{Brant C. Jones}
\address{Department of Mathematics, One Shields Avenue, University of California, Davis, CA 95616}
\email{\href{mailto:brant@math.ucdavis.edu}{\texttt{brant@math.ucdavis.edu}}}
\urladdr{\url{http://www.math.ucdavis.edu/\~brant/}}

\thanks{The author received support from NSF grant DMS-0636297.}


\date{\today}

\maketitle

%%%%%%%%%%%%%%%%%%%%%%%%%%%%%%%%%%%%%%%%%%%%%%%%%%%%%%%%%%%%%%%%%%%%%
%  Begin content
%%%%%%%%%%%%%%%%%%%%%%%%%%%%%%%%%%%%%%%%%%%%%%%%%%%%%%%%%%%%%%%%%%%%%

%%%%%%%%%%%%%%%%%%%%%%%%%%%%%%%%%%%%%%%%%%%%%%%%%%%%%%%%%%%%%%%%%%%%%
%  Section
%%%%%%%%%%%%%%%%%%%%%%%%%%%%%%%%%%%%%%%%%%%%%%%%%%%%%%%%%%%%%%%%%%%%%

\bigskip
\section{Introduction}\label{s:background}

The Bruhat partial order on the elements of a Coxeter group is a fundamental tool
in algebraic combinatorics, representation theory and the geometry of Schubert
varieties.  In this work, we give a derivation of the M\"obius function for this
partial order based on an explicit nonrecursive matching of the Hasse diagram.
The M\"obius function is used to invert formulas defined by sums over Bruhat
intervals, and gives the Euler characteristic in poset topology.  Many proofs
of the M\"obius function have appeared in the literature; see
\cite{verma,deodhar_mobius,k-l,b-w-82,stembridge_mobius,marietti-zircon}.  

Our construction is closest to Verma's original argument, although it is phrased
in terms of combinatorial objects called masks that are related to
Kazhdan--Lusztig combinatorics.  In \cite{verma}, Verma constructs a complete
matching of the Hasse diagram of the Bruhat interval $[x,w]$ in ``half'' the
cases:  when there exists a Coxeter generator $s_i$ such that $x s_i > x$ and $w
s_i < w$.  In the other cases, he applies an inductive argument to prove the
M\"obius function formula, but this argument does not extend to give a complete
matching of the Bruhat interval.  The complete matching that we give below can
be seen to agree with Verma's in the case that there exists $s_i$ satisfying $x
s_i > x$ and $w s_i < w$.  This case is also an example of a special matching
that has been used to compute $R$-polynomials in Kazhdan--Lusztig theory; see
\cite[Proposition 5.6.1]{b-b}.  In addition, \cite{rietsch-williams} have used
a complete matching of the intervals in finite Coxeter groups in order to apply
discrete Morse theory to totally nonnegative flag varieties.  We show that our
construction also agrees with this matching for the case of finite Coxeter
groups.  

Our matching unifies these constructions, and has the advantage of being given
explicitly and nonrecursively.  It also extends to intervals in infinite
Coxeter groups.

%%%%%%%%%%%%%%%%%%%%%%%%%%%%%%%%%%%%%%%%%%%%%%%%%%%%%%%%%%%%%%%%%%%%%
%  Section
%%%%%%%%%%%%%%%%%%%%%%%%%%%%%%%%%%%%%%%%%%%%%%%%%%%%%%%%%%%%%%%%%%%%%

\bigskip
\section{Construction}\label{s:construction}

Let $W$ be a Coxeter group with generating set $S$ and relations of the form
$(s_i s_j)^{m(i,j)} = 1$.  An \em expression \em is any product of generators
from $S$ and the \em length \em $l(w)$ is the minimum length of any expression
for the element $w$.  Such a minimum length expression is called \em reduced\em.
Given $w \in W$, we represent reduced expressions for $w$ in sans serif font,
say $\w = \w_1 \w_2 \cdots \w_p$ where each $\w_i \in S$.  For any $x, w \in W$,
we say that \em $x \leq w$ in Bruhat order \em if a reduced expression for $x$
appears as a subword (that is not necessarily consecutive) of some reduced
expression for $w$.  There are several other characterizations of this partial
order on the elements of $W$; see \cite{h,b-b} for details.  If $s_i$ appears
as the last (first, respectively) factor in some reduced expression for $w$,
then we say that $s_i$ is a \em right (left, respectively) descent \em for $w$;
otherwise, $s_i$ is an \em right (left, respectively) ascent \em for $w$.  If
$s_i$ is a descent for an element $w$ with reduced expression $\w = \w_1 \w_2
\cdots \w_p$ then the \em Exchange Condition \em implies that there exists an
index $i$ for which $w s_i = \w_1 \cdots \w_{i-1} \widehat{\w_{i}} \w_{i+1}
\cdots \w_p$, where the hat indicates omission.

The following lemma gives a useful property of Bruhat order.

\begin{lemma}{\bf (Lifting Lemma) \cite[Proposition 2.2.7]{b-b}}\label{l:lifting}
Suppose $x < w$, $s_i$ is a right descent for $w$, and $s_i$ is a right ascent for $x$.
Then, $x s_i \leq w$ and $w s_i \geq x$.
\end{lemma}

In this work, we will represent Bruhat relations using a combinatorial model
inspired by Deodhar \cite{d} and Billey--Warrington \cite{b-w} for the purpose
of studying Kazhdan--Lusztig polynomials.  Fix a reduced expression $\w = \w_1
\w_2 \cdots \w_p$.  Define a \em mask \em $\s$ associated to the reduced expression $\w$ to be any
binary vector $(\s_1, \ldots, \s_p)$ of length $p = l(w)$.  Every mask
corresponds to a subexpression of $\w$ defined by $\w^\s =
\w_{1}^{\s_1} \cdots \w_{p}^{\s_p}$ where
\[
\w_{j}^{\s_j}  =
\begin{cases}
\w_{j}  &  \text{ if  }\s_j=1\\
\text{1}  &  \text{ if  }\s_j=0.
\end{cases}
\]
Each $\w^\s$ is a product of generators so it determines an element of $W$.  For
$1\leq j\leq p$, we also consider initial sequences of a mask denoted $\s[j] =
(\s_1, \ldots, \s_j)$, and the corresponding initial subexpression $\w^{\s[j]}
= \w_{1}^{\s_1} \cdots \w_{j}^{\s_j}$.  In particular, we have $\w^{\s[p]} =
\w^\s$.  We also use this notation to denote initial sequences of expressions,
so $\w[j] = \w_1 \cdots \w_j$.

We say that a position $j$ (for $2 \leq j \leq p$) of the fixed reduced
expression $\w$ is a \em defect \em with respect to the mask $\s$ if
\begin{equation*}
\w^{\s[j-1]} \w_{j} < \w^{\s[j-1]}.
\end{equation*}
Note that the defect status of position $j$ does not depend on the value of
$\s_j$.  
We say that a defect position is a \em 0-defect \em if it has mask-value 0, and
call it a \em 1-defect \em if it has mask-value 1.
If a mask has no defect positions at all, then we say it is a \em constant mask
on the reduced expression $\w$ for the element $\w^{\s}$\em.  This terminology
arises from the fact that these masks correspond precisely to the unique
constant term in the Kazhdan--Lusztig polynomial $P_{x,w}(q)$ in the
combinatorial model mentioned above.  Other authors
\cite{marsh-rietsch,rietsch-williams} have used the term ``positive
distinguished subexpression'' to define an equivalent notion.

The following result is due to Deodhar \cite[Proposition 2.3(iii)]{d}, and has
also appeared in work of \cite{marsh-rietsch} related to totally nonnegative
flag varieties, as well as \cite{armstrong} in the context of sorting algorithms
on Coxeter groups.  As the lemma is central to our work, we include a proof here
for completeness.

\begin{lemma}{\bf (Deodhar)}\label{l:unique_constant_mask}
Let $\w = \w_1 \cdots \w_p$ be a reduced expression for an element $w \in W$ and
let $x \leq w$.  Then there is a unique constant mask $\s$ on $\w$ for $x$.
\end{lemma}
\begin{proof}
We describe a greedy algorithm to construct such a mask.  
Let $r_{p+1}(x) = x$ and $i = p$.  We inductively assign
\[ \s_i := \begin{cases}
    0 &  \text{ if $\w_i$ is a right ascent for $r_{i+1}(x)$ } \\
    1 &  \text{ if $\w_i$ is a right descent for $r_{i+1}(x)$ } \\
\end{cases} \ \ 
\]
\[
\text{ and } \ \ \ \ r_i(x) := \begin{cases}
    r_{i+1}(x) &  \text{ if $\w_i$ is a right ascent for $r_{i+1}(x)$ } \\
    r_{i+1}(x) \cdot \w_i &  \text{ if $\w_i$ is a right descent for $r_{i+1}(x)$ } \\
\end{cases} 
\]
for each $i$ from $p$ down to $1$.  

Note that the constraint that $\s$ have no defects forces the choice of
mask-value at each step.  Hence, there can be at most one mask $\s$ on $\w$ for
$x$.  In particular, the algorithm produces a constant mask on $\w$ for $x$ if
and only if $r_1(x)$ is the identity.

We claim that $r_1(x)$ is always the identity.  Note that we have a constant
mask consisting of all 1 entries for $x = w$.  Hence, if $x < w$ and we run the
algorithm for both elements simultaneously, we initially have $r_{p+1}(x) = x
\leq w = r_{p+1}(w)$.  Observe that for each $i \leq p$, whenever we have
$r_{i+1}(x) \leq r_{i+1}(w)$ then $r_{i}(x) \leq r_{i}(w)$.  This follows by
definition when $r_{i}(w) = r_{i+1}(w)$, and by an application of the Lifting
Lemma~\ref{l:lifting} in the case that $r_{i+1}(w)$ covers $r_{i}(w)$ using the
fact that $r_i(x)$ always has $\w_i$ as a right ascent by construction.  Since
$r_1(w) = 1$, this implies by induction that $r_1(x) = 1$ so the algorithm
produces a constant mask for all $x<w$.
\end{proof}

\begin{example}
If $W = A_4$, $\w = s_2 s_3 s_4 s_1 s_2 s_3$ and $x = s_1 s_2 s_1$ then $\s$ is
\[
\begin{tabular}{cccccc}
    $s_2$ & $s_3$ & $s_4$ & $s_1$ & $s_2$ & $s_3$ \\
    1 & 0 & 0 & 1 & 1 & 0 \\
\end{tabular}
\]
as a result of 
\[ r_7(x) = s_1 s_2 s_1 = r_6(x), r_5(x) = s_2 s_1, r_4(x) = s_2 = r_3(x) = r_2(x), r_1(x) = 1. \]
\end{example}

\begin{remark}
Suppose $\s$ and $\tau$ are constant masks on a fixed reduced expression $\w$.
Then it can be shown that the mask $\nu = \s \vee \tau$ defined by
\[ \nu_i := \begin{cases}
    1 &  \text{ if $\s_i = 1$ or $\tau_i = 1$, } \\
    0 &  \text{ otherwise. }
\end{cases} \]
is a constant mask.  Although Bruhat order is not a lattice, the operation
$\vee$ can be used to define an associated join-semilattice that respects
Bruhat order, once we fix a reduced expression $\w$.

In fact, \cite{armstrong} has shown that this partial order on $[1, \w]$ is a
lattice that lies maximally between the weak and strong Bruhat orders on $W$.
\end{remark}

Continue to fix the reduced expression $\w = \w_1 \cdots \w_p$ for $w \in W$ and
suppose $y \leq x \leq w$.  We describe a notion of relative mask that captures
this pair of Bruhat relations.  Let $\tau$ be the unique constant mask on $\w$
for $x$.  Then, $\w^{\tau}$ is a reduced expression for $x$ and we may let
$\nu$ be the unique constant mask on $\w^{\tau}$ for $y$.  We combine these
into a \em relative mask \em $\s = (\s_1, \ldots, \s_p)$ by
\[
\s_j  =
\begin{cases}
X & \text{ if $\tau_j = 0$ } \\
\nu_j & \text{ if $\tau_j = 1$. }
\end{cases}
\]
In this situation, we call $\tau$ the \em $X$-mask \em associated to $(\w, \s)$,
also denoted $\Xi(\s)$.  We denote $(\w^{\tau})^{\nu}$ by $\w^{\s}$.  We say
that position $j$ is a \em defect \em in the relative mask $\s$ if $\w^{\s[j-1]}
\w_j < \w^{\s[j-1]}$.  Note that only positions in $\s$ with mask-value $X$ can
be defects, by definition.  We will indicate these defect positions by $X^d$ in
our illustrations of relative masks.

\begin{example}
The relative masks encoding the Bruhat interval $[s_2, s_2 s_1 s_3 s_2]$ in type
$A$ are given by
\[
\begin{tabular}{ccccccccccc}
   & $s_2$ & $s_1$ & $s_3$ & $s_2$ & & $s_2$ & $s_1$ & $s_3$ & $s_2$ & subexpression for $x \in [s_2, s_2 s_1 s_3 s_2]$ \\
   $\s =$ & 0 & 0 & 0 & 1 & \ \ \ \ $\tau = $ & 1 & 1 & 1 & 1 & $s_2 s_1 s_3 s_2$ \\
   $\s = $ & 1 & 0 & 0 & $X^d$ & \ \ \ \ $\tau = $ & 1 & 1 & 1 & 0 & $s_2 s_1 s_3$ \\
   $\s = $ & 0 & 0 & $X$ & 1 & \ \ \ \ $\tau = $ & 1 & 1 & 0 & 1 & $s_2 s_1 s_2$ \\
   $\s = $ & 0 & $X$ & 0 & 1 & \ \ \ \ $\tau = $ & 1 & 0 & 1 & 1 & $s_2 s_3 s_2$ \\
   $\s = $ & $X$ & 0 & 0 & 1 & \ \ \ \ $\tau = $ & 0 & 1 & 1 & 1 & $s_1 s_3 s_2$ \\
   $\s = $ & 1 & 0 & $X$ & $X^d$ & \ \ \ \ $\tau = $ & 1 & 1 & 0 & 0 & $s_2 s_1$ \\
   $\s = $ & 1 & $X$ & 0 & $X^d$ & \ \ \ \ $\tau = $ & 1 & 0 & 1 & 0 & $s_2 s_3$ \\
   $\s = $ & $X$ & 0 & $X$ & 1 & \ \ \ \ $\tau = $ & 0 & 1 & 0 & 1 & $s_1 s_2$ \\
   $\s = $ & $X$ & $X$ & 0 & 1 & \ \ \ \ $\tau = $ & 0 & 0 & 1 & 1 & $s_3 s_2$ \\
   $\s = $ & $X$ & $X$ & $X$ & 1 & \ \ \ \ $\tau = $ & 0 & 0 & 0 & 1 & $s_2$ \\
\end{tabular}
\]
Here, $\w^{\s} = s_2$ for all of these masks.
\end{example}

Our goal is to give a matching on the Hasse diagram of the Bruhat interval
$[y,w]$ using the relative masks for $y$ on a fixed reduced expression for $w$
as an encoding.  The following definition will allow us to define a procedure
that is reversible.

\begin{definition}\label{d:shifted_descent}
Let $\w = \w_1 \cdots \w_p$ be a reduced expression for an element $w \in W$ and let $\s$ be a
relative mask on $\w$ with $X$-mask $\tau$.  We call position $j$ a \em shifted descent \em of
$(\w, \s)$ if $\w^{\tau[j-1]} \geq \w^{\s[j]}$ in Bruhat order.
\end{definition}

\begin{example}
Consider
\[
\begin{tabular}{ccccccccccc}
   & $s_2$ & $s_1$ & $s_3$ & $s_2$ & & $s_2$ & $s_1$ & $s_3$ & $s_2$ & subexpression for $x \in [s_2, s_2 s_1 s_3 s_2]$ \\
   $\s = $ & 0 & $X$ & 0 & 1 & \ \ \ \ $\tau = $ & 1 & 0 & 1 & 1 & $s_2 s_3 s_2$ \\
   $\s = $ & $X$ & 0 & 0 & 1 & \ \ \ \ $\tau = $ & 0 & 1 & 1 & 1 & $s_1 s_3 s_2$ \\
\end{tabular}
\]
The first mask has position 4 as a shifted descent because $s_2 s_3 \geq s_2$.
The second mask does not have position 4 as a shifted descent because $s_1 s_3
\ngeq s_2$.
\end{example}

We are now in a position to define our matching.  The rough idea that motivates
the following definition is to remove the rightmost $X$ from a relative mask
$\s$ in a way that preserves $\w^{\s}$ and is reversible.

\begin{definition}\label{d:matching}
Let $\s$ be a relative mask on $\w$ with $X$-mask $\tau$.  Find the rightmost
position $j$ in $(\w, \s)$ where one of the following conditions holds, and
apply the given transformation to obtain a new relative mask denoted $\varphi(\s)$:
\begin{itemize}
    \item[(1)]  If $\s_j = X$ and $\s_j$ is not a defect then change $\s_j$ to
        0.
    \item[(2)]  If $\s_j = 0$ then change $\s_j$ to $X$.  Note that by
        definition, $\s_j$ cannot be a defect in this case.
    \item[(3)]  If $\s_j = X$ and $\s_j$ is a defect then $\w^{\tau[j-1]} \geq
        \w^{\s[j-1]} > \w^{\s[j-1]} \w_j$.  Hence, we may assign the unique
        constant mask for $\w^{\s[j-1]} \w_j$ on $\w^{\tau[j-1]}$ to the entries
        of $\w^{\tau[j-1]}$ and set $\s_j$ to 1.
    \item[(4)]  If $\s_j = 1$ and $\s_j$ is a shifted descent then $\w^{\tau[j-1]} \geq
        \w^{\s[j]}$ so we may assign the unique constant mask for $\w^{\s[j]}$
        on $\w^{\tau[j-1]}$ to the entries of $\w^{\tau[j-1]}$ and set $\s_j$ to
        $X$.  Note that by definition, $\s_j$ becomes a defect in this case.
\end{itemize}
\end{definition}

\begin{example}\label{e:interval_matching}
The matching given by $\varphi$ on $[s_2, s_2 s_1 s_3 s_2]$ is:

\[
\xymatrix {
& s_2 s_1 s_3 s_2 \cong [ 0 0 0 1 ] \ar@{-}[dl] & & \\
s_2 s_1 s_3 \cong [ 1 0 0 X^d ] & s_2 s_3 s_2 \cong [ 0 X 0 1 ] \ar@{-}[dl] & s_2 s_1 s_2 \cong [ 0 0 X 1 ] \ar@{-}[dl] & s_1 s_3 s_2 \cong [ X 0 0 1 ] \ar@{-}[dl] \\
s_2 s_3 \cong [ 1 X 0 X^d ] & s_2 s_1 \cong [ 1 0 X X^d ] & s_1 s_2 \cong [ X 0 X 1] & s_3 s_2 \cong [ X X 0 1 ] \ar@{-}[dl] \\
  & & s_2 \cong [ X X X 1 ] & \\
}
\]
\end{example}

We now give our main results.

\begin{lemma}\label{l:valid}
The function $\varphi$ given in Definition~\ref{d:matching} always produces a
valid relative mask.
\end{lemma}
\begin{proof}
Let $\s$ be a relative mask on $\w$.  Applying $\varphi$ to $\s$ interchanges
exactly one $X$ entry in position $j$ with an entry that is either 0 or 1, and
also possibly rearranges the 0 and 1 entries lying to the left of $j$.  By
definition, we never create a 0-defect nor 1-defect at $j$, and applying any of
the rules at position $j$ preserves the element $\w^{\s[j]}$, so the defect
status of positions $k > j$ does not change.  Hence, to show that $\varphi(\s)$ is
valid it suffices to show that $\varphi(\s)$ has a constant $X$-mask.
Let $\tau$ denote the $X$-mask of $\s$.

Suppose $\varphi$ acts at position $j$ by changing $\s_j$ from $X$ to 1 or 0.
Then the $X$-mask $\tau'$ of $\varphi(\s)$ is the result of changing the
rightmost 0 to a 1 in the $X$-mask $\tau$ of $\s$.  Since $\tau$ is a constant
mask, we have that $\tau'$ is constant by \cite[Proposition 5.3.9]{b-b}.

Now consider the case where $\varphi$ acts at position $j$ by changing $\s_j$
from 1 or 0 to $X$, and suppose for the sake of contradiction that the $X$-mask
$\tau'$ of $\varphi(\s)$ is not constant.  By Definition~\ref{d:matching}, we
have $\s_i = 1$ for all $i > j$, for otherwise we would have applied $\varphi$
to position $i$.  Since $\tau'$ is not constant, there exists a leftmost
position $k > j$ that becomes a 1-defect in $\tau'$.  Hence, we have the
schematic shown below.
\[
\tiny
\xymatrix @-2pc {
& 1 & & \cdots & & j & & \cdots & & (k-1) & k \\
\s = & \ast & X & \cdots & \ast & \ast & 1 & \cdots & 1 & 1 & 1 \\
\tau' = & 1 & X & \cdots & 1 & X & 1 & \cdots & 1 & 1 & 1^d \\
}
\]
Here, the positions marked by $\ast$ are the non-$X$ positions of $\s$, so these
positions have mask-value 1 in $\tau'$.

If $\s_j = 0$, then $\w^{\s[k-1]} \leq \w^{\tau'[k-1]}$ and $\w_k$ is a
descent for $\w^{\tau'[k-1]}$ while $\w_k$ is an ascent for $\w^{\s[k-1]}$.  Hence, by
the Lifting Lemma~\ref{l:lifting} we have $\w^{\s[k]} \leq \w^{\tau'[k-1]} \leq
\w^{\tau[k-1]}$ so $k$ is a shifted descent in $\s$, contradicting the
rightmost choice of move in Definition~\ref{d:matching}.

Next, suppose $\s_j = 1$, as shown in the schematic below.
\[
\tiny
\xymatrix @-2pc {
& 1 & & \cdots & & j & & \cdots & & (k-1) & k \\
\s = & \ast & X & \cdots & \ast & 1 & 1 & \cdots & 1 & 1 & 1 \\
\varphi(\s) = & \ast & X & \cdots & \ast & X^d & 1 & \cdots & 1 & 1 & 1 \\
\tau' = & 1 & X & \cdots & 1 & X & 1 & \cdots & 1 & 1 & 1^d \\
}
\]
Then, since $\varphi$ operates on $\s$ at position $j$, we
have that $j$ is a shifted descent.  Therefore, position $j$ in the mask
$\varphi(\s)$ is $X^d$ and $\w^{\varphi(\s)[j]} = \w^{\s[j]}$.  

Hence, we have $\w^{\varphi(\s)[k-1]} \leq \w^{\tau'[k-1]}$ because we have
exhibited one as a submask of the other without 1-defects.  Moreover, $\w_k$ is a descent
for $\w^{\tau'[k-1]}$ and an ascent for $\w^{\varphi(\s)[k-1]}$, so by the
Lifting Lemma~\ref{l:lifting},
we have $\w^{\varphi(\s)[k-1]} \w_k \leq \w^{\tau'[k-1]}$.  Hence,
\[ \w^{\s[k]} = \w^{\s[k-1]} \w_k = \w^{\varphi(\s)[k-1]} \w_k \leq
\w^{\tau'[k-1]} \leq \w^{\tau[k-1]} \]
so $k$ was a shifted descent in $\s$ to begin with, contradicting the rightmost
choice of move in Definition~\ref{d:matching}.

Since all cases where $\tau'$ is not constant lead to a contradiction, we have
completed the proof that $\varphi(\s)$ is a valid relative mask.
\end{proof}

\begin{lemma}\label{l:invol}
The function $\varphi$ given in Definition~\ref{d:matching} is an
involution on the set of relative masks on $\w$.
\end{lemma}
\begin{proof}
To see that $\varphi$ is an involution, we observe that rule (2) inverts rule
(1) in Definition~\ref{d:matching}, and it is straightforward to verify that
that rule (4) inverts rule (3).  Moreover, applying any of the rules at position
$j$ preserves the element $\w^{\s[j]}$, so the mask-value and defect status of
positions $k > j$ does not change.  

Since the rules in Definition~\ref{d:matching} depend only on mask-value, defect
status and shifted descent status, the only way in which applying a move at
position $j$ can create a new move at position $k > j$ is if $k$ becomes a
shifted descent in $\varphi(\s)$.  Moreover, this can only occur as a result of
applying rules (1) or (3) to $\s$.

For the sake of contradiction suppose this occurs and among all counterexamples,
consider one such that $l(w)$ is minimal.  Then, $\varphi(\s)$ operates at
position $j$ and $\varphi(\varphi(\s))$ operates at position $k > j$.  
Let $\tau$ be the $X$-mask of $\s$ and $\tau'$ be the $X$-mask of
$\varphi^2(\s)$.
Thus, we have the following schematic.
\[
\tiny
\xymatrix @-2.2pc {
                & 1 & \cdots & j & \cdots & i & \cdots & k-1 & k \\
\s =            & \ast & \cdots & X^{(d)} & \cdots & 1 & \cdots & 1 & 1   \\
\varphi(\s) =   & \ast & \cdots & \ast & \cdots & 1 & \cdots & 1 & 1   \\
\varphi^2(\s) = & \ast & \cdots & 1    & \cdots & 1 & \cdots & 1 & X^d   \\
}
\]
We begin by justifying the main points of this schematic.  Note that we do not
assume the defect status of $\s_j$ is known.  If $\s_j$ is a defect then
$\varphi(\s)_j = 1$, and if it is not then $\varphi(\s)_j = 0$.  In any case,
observe that $\varphi^2(\s)_j$ must be 1, for otherwise the mask $\varphi^2(\s)$
shows that position $k$ is already a shifted descent in $\s$.  Also, observe
that all the entries between $j$ and $k$ have mask-value 1 in $\s$ and
$\varphi(\s)$, for otherwise we contradict that $j$ is the rightmost move in
$\s$.
By assuming that $\w$ is minimal length, we have that there are no other $X$
entries in any of the masks because if there exists an $X$ entry in one of the
masks, it exists in all three of the masks, by virtue of the fact that we only
adjust the $X$-masks at positions $j$ and $k$ as shown.  Hence, any
$X$-positions could be removed from all three masks simultaneously.

Next, we consider all possible cases of mask-values for $\s$ and $\varphi^2(\s)$
on $\w_1$.  

{\bf Case:}  ($\s_1 = 0$ and $\varphi^2(\s)_1 = 0$) or ($\s_1 = 1$ and
$\varphi^2(\s)_1 = 1$).  Here, we have
\[
\tiny
\xymatrix @-2.2pc {
                & 1 & \cdots & j & j+1 & \cdots & k-1 & k \\
\s =            & 0 & \cdots & X^{(d)} & 1 & \cdots & 1 & 1   \\
%\varphi(\s) =   & 0 & \cdots & \ast & 1 & \cdots & 1 & 1   \\
\varphi^2(\s) = & 0 & \cdots & 1 & 1 & \cdots & 1 & X^d   \\
} \ \ \ \ \ \ \ \  \parbox{0.3in}{\vspace{0.25in} or } 
\xymatrix @-2.2pc {
                & 1 & \cdots & j & j+1 & \cdots & k-1 & k \\
\s =            & 1 & \cdots & X^{(d)} & 1 & \cdots & 1 & 1   \\
%\varphi(\s) =   & \ast & \cdots & \ast & 1 & \cdots & 1 & 1   \\
\varphi^2(\s) = & 1 & \cdots & 1 & 1 & \cdots & 1 & X^d   \\
}
\]
Let $\nu$ denote the restriction of $\s$ to $\w_2 \cdots \w_k$.  Then the
restriction of $\varphi^2(\s)$ to $\w_2 \cdots \w_k$ shows that $k$ becomes a
shifted descent in $\varphi(\nu)$.  Therefore, we obtain a counterexample on
$\w_2 \cdots \w_k$, contradicting our minimal length choice of $w$.

{\bf Case:}  $\s_1 = 0$ and $\varphi^2(\s)_1 = 1$.
In this case, $\w_1$ is a left descent for $\w^{\s}$.
\[
\tiny
\xymatrix @-2.2pc {
                & 1 & \cdots & j & j+1 & \cdots & k-1 & k \\
\s =            & 0 & \cdots & X^{(d)} & 1 & \cdots & 1 & 1   \\
%\varphi(\s) =   & \ast & \cdots & \ast & 1 & \cdots & 1 & 1   \\
\varphi^2(\s) = & 1 & \cdots & 1 & 1 & \cdots & 1 & X^d   \\
}
\]
By the Exchange Condition, there exists some position $i$ such that $\w_1
\w^{\s} = \w^{\nu}$ where $\nu$ is obtained from $\s$ by changing a single
mask-value 1 entry at $\w_i$ to have mask-value 0.  Observe that if $i > j$ then
$\w_1 \w^{\s[i-1]} = \w^{\s[i]}$ so changing $\s_1$ to 1 would witness that $i$
was a shifted descent in $\s$, a contradiction.

The mask $\nu$ may be not be a constant mask on $\w^{\tau}$, but there exists a
unique constant mask $\gamma$ for the element $\w^{\nu}$ on $\w_2^{\tau_2}
\cdots \w_k^{\tau_k}$ and $\gamma$ still has mask-value 1 on $\w_{j+1} \cdots
\w_k$ by the algorithm from Lemma~\ref{l:unique_constant_mask}.  By abuse of
notation, let $\gamma$ denote the corresponding relative mask on $\w_2 \cdots
\w_k$.

If there exists a shifted descent in position $m$ of $\gamma$ where $j < m \leq
k$, then $m$ must have been a shifted descent in $\s$, a contradiction.  To see
this, observe that $\w^{\s[m]} = \w_1 \w^{\gamma[m]}$ because $\w^{\s} = \w_1
\w^{\gamma}$ and these reduced expressions agree in positions $m, \ldots, k$.
Therefore, if $\w^{\gamma[m]} \leq \w^{\tau[m-1]}$ then $\w_1 \w^{\s[m]} \leq
\w^{\tau[m-1]}$ and by the Lifting Lemma, we have $\w^{\s[m]} \leq
\w^{\tau[m-1]}$.

Then, the restriction of $\varphi^2(\s)$ to $\w_2 \cdots \w_k$ shows that $k$
becomes a shifted descent in $\varphi(\gamma)$.  Therefore, we obtain a
counterexample on $\w_2 \cdots \w_k$, contradicting our minimal length choice of
$w$.

{\bf Case:}  $\s_1 = 1$ and $\varphi^2(\s)_1 = 0$.
In this case, $\w_1$ is a left descent for $\w^{\varphi^2(\s)}$.
\[
\tiny
\xymatrix @-2.2pc {
                & 1 & \cdots & j & j+1 & \cdots & k-1 & k \\
\s =            & 1 & \cdots & X^{(d)} & 1 & \cdots & 1 & 1   \\
%\varphi(\s) =   & \ast & \cdots & \ast & 1 & \cdots & 1 & 1   \\
\varphi^2(\s) = & 0 & \cdots & 1 & 1 & \cdots & 1 & X^d   \\
}
\]
By the Exchange Condition, there exists some position $i$ such that $\w_1
\w^{\varphi^2(\s)} = \w^{\nu}$ where $\nu$ is obtained from $\varphi^2(\s)$ by
changing a single mask-value 1 entry to have mask-value 0.  
The mask $\nu$ may be not be a constant mask on $\w_2 \cdots \w_{k-1}$, but
there exists a constant mask $\gamma$ on $\w_2 \cdots \w_{k-1}$ for the element
$\w^{\nu}$ by Lemma~\ref{l:unique_constant_mask}.  Let $\rho$ denote $\s$
restricted to $\w_2 \cdots \w_k$.
Then, $\gamma$ shows that $k$ becomes a shifted descent in $\varphi(\rho)$.
Therefore, we obtain a counterexample on $\w_2 \cdots \w_k$, contradicting our
minimal length choice of $w$.

This in all cases, we have shown that applying a move at position $j$ cannot
create a new move at some position $k > j$.  
Hence, $\varphi$ is an involution.
\end{proof}

\begin{theorem}\label{t:main}
The function $\varphi$ given in Definition~\ref{d:matching} is a complete
matching of the Hasse diagram of Bruhat order on $[y,w]$ whenever $y < w$.
\end{theorem}
\begin{proof}
Encode $[y, w]$ as a set of relative masks for $y$ on a fixed reduced
expression $\w$ for $w$, so each $x \in [y,w]$ is given by $\w^{\Xi(\s)}$ where
$\Xi(\s)$ is the $X$-mask of a relative mask $\s$.  By Lemmas~\ref{l:valid} and
\ref{l:invol}, the function $\varphi$ given in Definition~\ref{d:matching} is
an involution that interchanges exactly one $X$ entry in each relative mask for
an entry that is either 0 or 1.  Since the $X$-masks of both elements are
constant masks, $l(w)-l(x)$ is given by the number of $X$ entries in the
relative mask, so this operation represents a cover relation in Bruhat order.
Hence, we have that $\varphi$ is a matching on the Hasse diagram of $[y,w]$.
Unmatched relative masks must contain no $0$ entries nor $X$ entries at all, so
consist of all 1 entries, and this occurs only if $w = y$.  Hence, the matching
is complete when $y < w$, and the result follows.
\end{proof}

\begin{corollary}\label{c:main}
The M\"obius function of the Bruhat interval $[y,w]$ is $\mu(x,w) =
(-1)^{l(w)-l(x)}$.
\end{corollary}
\begin{proof}
Following Verma \cite{verma}, it suffices to show that there exists a complete
matching of the Hasse diagram of the Bruhat interval $[y,w]$ whenever $y < w$.
This matching can then be interpreted as a sign-reversing involution on 
\[ \sum_{y \leq x \leq w} (-1)^{l(w) - l(x)} \]
proving that the sum is equal to the Kronecker function $\delta_{y, w}$.
This follows from Theorem~\ref{t:main}.
\end{proof}

In work related to totally nonnegative flag varieties, \cite{rietsch-williams}
have given another complete matching on Bruhat intervals $[y,w]$ of a finite
Coxeter group $W$.  This matching $M$ is defined recursively, starting from an
EL-labeling of the interval and a chosen reduced expression $\w$ for $w$.  To
describe this matching, we begin with a reduced expression $\w_0$ for the
longest element $w_0$ of $W$ having $\w^{-1} = \w_p \w_{p-1} \cdots \w_1$ as a
left factor.  Then we obtain a total ordering on the reflections of $W$ using
the inversion sequence constructed from the reduced expression $\w_0$ by
\begin{equation}\label{e:reflection_order}
\w_p > \w_{p} \w_{p-1} \w_{p} > \cdots > \w_{p} \w_{p-1} \cdots \w_{p-i+1} \w_{p-i} \w_{p-i+1} \cdots \w_{p-1} \w_{p} > \cdots .
\end{equation}
We label all of the Bruhat cover relations $x' \lessdot x$ in $[y,w]$ by the
unique right reflection $t$ such that $x' = xt$.  Then, Dyer \cite{dyer3} has
shown that this is an EL-labeling.  Rietsch and Williams construct a matching
$M$ from this EL-labeling using a result of Chari \cite{chari}.

\begin{remark}\label{r:lw}
For finite Coxeter groups, we show that the matching $M$ is the same as the
matching given in Theorem~\ref{t:main}, working by downward induction on the
ranks of the partial order $[y,w]$.  We begin at the top rank $r$ containing
$w$.

Let $x$ be an unmatched element on the current rank $r$.  Consider the relative
mask $\s$ associated to $x$.  Since $x$ is a maximal unmatched element, when we
apply $\varphi$ to $\s$, we operate by placing an $X$ in the rightmost position
$i$ such that the element $\w^{\tau}$ associated to the resulting $X$-mask
$\tau = \Xi(\varphi(\s))$ still contains $y$ in Bruhat order.  Moreover,
observe that none of the entries to the right of $i$ are shifted descents, nor
do they have mask values $X$ or 0, for otherwise $\s$ would already have been
matched.  Hence, $\w^{\tau}$ is the element
\[ x \cdot (\w_p \w_{p-1} \cdots \w_{i+1} \w_i \w_{i+1} \cdots \w_{p-1} \w_{p}), \] 
and any element $x \cdot (\w_p \w_{p-1} \cdots \w_{j+1} \w_j \w_{j+1} \cdots
\w_{p-1} \w_{p})$ for $j > i$ does not contain $y$ in Bruhat order, for
otherwise $j$ would be a shifted descent in $\s$.

In comparison, \cite[Corollary 7.8]{rietsch-williams} states that the matched
edge $x' \lessdot x$ in $M$ has the largest EL-label in the sense of
Equation~(\ref{e:reflection_order}) among all of the edges descending from $x$
in $[y,w]$.  But this is equivalent to the rightmost condition that we used to
choose $i$.  Hence, we see that $\w^{\tau}$ is equal to the element $x'$ that
is matched to $x$ in $M$.  This proves that the matchings agree on all elements
down to rank $r$, and we can proceed to apply the argument to the unmatched
elements on rank $r-1$.  Continuing in this fashion, we find that the matchings
agree on $[y,w]$.  
\end{remark}

The description given in Theorem~\ref{t:main} has the advantage of being
nonrecursive and also permits some observations that are perhaps less clear in
the other language.  For example, we see that the matched edges of $M$ are
always labeled by one of the reflections that represent inversions in $\w$, so
the matching does not depend on how $\w$ is completed to a reduced expression
for $w_0$.

%%%%%%%%%%%%%%%%%%%%%%%%%%%%%%%%%%%%%%%%%%%%%%%%%%%%%%%%%%%%%%%%%%%%%
%  Section
%%%%%%%%%%%%%%%%%%%%%%%%%%%%%%%%%%%%%%%%%%%%%%%%%%%%%%%%%%%%%%%%%%%%%
\section{Further questions}

Bruhat order extends to parabolic quotients of Coxeter groups as described in
\cite[Section 2.5]{b-b}.  Deodhar has given a parabolic version of the M\"obius
function formula in Deodhar~\cite[Theorem 1.2]{deodhar_mobius}, and it would be
interesting to extend the mask matching given above to recover his result.

Also, the order complex associated to a Bruhat interval $[x,w]$ is a
topological space known to be homeomorphic to the $(l(w)-l(x)-2)$-sphere.  It
would be interesting to recover the poset topology of the Bruhat intervals from
the combinatorial matching we have given above.

As a preliminary step in this direction, we have observed that our matching is
acyclic, in the sense used in discrete Morse theory.  When $W$ is a finite
Coxeter group, this could also be inferred from \cite{rietsch-williams} by
Remark~\ref{r:lw}.

\begin{definition}
Consider the Hasse diagram of Bruhat order as a directed graph with an edge $w
\rightarrow x$ if $w$ covers $x$.  Given a matching, reverse the direction of
each edge in the Hasse diagram corresponding to a matched edge.  We say the
matching is \em acyclic \em if there are no directed cycles in the resulting
directed graph.
\end{definition}

\begin{theorem}
The function $\varphi$ given in Definition~\ref{d:matching} is an acyclic
matching of the Hasse diagram of Bruhat order on $[y,w]$ whenever $y < w$.
\end{theorem}
\begin{proof}
Let $\w$ be a reduced expression for $w$ and consider the relative masks on $\w$
for $y$.  
Every directed cycle has at least two pairs of up-down edges.  Observe that each edge
pointing up corresponds to a matched edge, so is obtained by removing the
rightmost $X$ entry.  Each edge pointing down corresponds to a non-matched edge,
and this is a Bruhat cover on the elements encoded by the $X$-masks.

Recall that $\Xi(\s)$ denotes the $X$-mask of a relative mask $\s$ on $\w$.
Suppose we have a pair of up-down edges in a directed cycle
\[ x := \w^{\Xi(\s)} \rightarrow z := \w^{\Xi(\gamma)} \rightarrow x' :=
\w^{\Xi(\s')}. \]
Here, $z$ covers $x$ and $x'$ with $x \neq x'$, and $\gamma = \varphi(\s)$.  We claim that the
rightmost $X$-entry in $\s'$ occurs strictly left of the rightmost $X$-entry in
$\s$, which implies that there are no directed cycles.

Let $i$ denote the position in $\w$ where $\varphi$ acts on $\s$.  Since $i$ is
the rightmost move and $\varphi$ does not alter the mask-values to the right of
$i$, we must have $\gamma_j = 1$ for all $j > i$, and none of the positions
$\gamma_j$ are shifted descents for $j > i$.

Consider the rightmost $X$-entry in $\s'$ and suppose for the sake of
contradiction that it occurs in position $j \geq i$.  Then, the relative masks
$\gamma$ and $\s'$ agree on all positions strictly right of $j$ according to
the algorithm given in Lemma~\ref{l:unique_constant_mask} since we always
encode the same element $y$.  At position $j$, we have $\Xi(\gamma)[j] = 1$,
and $\Xi(\s')[j] = 0$.  Hence,
\begin{equation}\label{e:ac_sd}
    \w^{\gamma[j]} = \w^{\s'[j]} \leq \w^{\Xi(\s')[j-1]} = \w^{\Xi(\gamma)[j-1]}.
\end{equation}
To see the last equality, we use a Lifting Lemma argument.  We have
$\w^{\Xi(\s')[j-1]} = \w^{\Xi(\s')[j]} \leq \w^{\Xi(\gamma)[j]}$, and $\w_j$ is
a right descent for $\w^{\Xi(\gamma)[j]}$, but $\w_j$ is a right ascent for
$\w^{\Xi(\s')[j-1]}$.  So, $\w^{\Xi(\s')[j-1]} \leq \w^{\Xi(\gamma)[j]} \w_j =
\w^{\Xi(\gamma)[j-1]}$.  However, $\w^{\Xi(\s')[j-1]}$ and
$\w^{\Xi(\gamma)[j-1]}$ have the same length so they
must be equal.

Equation~(\ref{e:ac_sd}) proves that $j$ is a shifted descent in $\gamma$.  It
also shows that if $j = i$, then $x = x'$.  In any case, we reach a
contradiction.  Hence, the rightmost $X$-entry in $\s'$ occurs strictly left of
the rightmost $X$-entry in $\s$, so the matching is acyclic.
\end{proof}

This is consistent with the main result of \cite{bjorner-regcw}.

%%%%%%%%%%%%%%%%%%%%%%%%%%%%%%%%%%%%%%%%%%%%%%%%%%%%%%%%%%%%%%%%%%%%%
%  Section
%%%%%%%%%%%%%%%%%%%%%%%%%%%%%%%%%%%%%%%%%%%%%%%%%%%%%%%%%%%%%%%%%%%%%

\bigskip
\section*{Acknowledgments}

We thank Drew Armstrong, Eric Babson, Francesco Brenti, Patricia Hersh, Sam
Hsiao, Nathan Reading, Konstanze Rietsch, Anne Schilling, and Monica Vazirani
for helpful conversations related to this work.  We especially thank Lauren
Williams for reading a draft of this paper and providing several excellent
suggestions.

%%%%%%%%%%%%%%%%%%%%%%%%%%%%%%%%%%%%%%%%%%%%%%%%%%%%%%%%%%%%%%%%%%%%%
% ==========   Bibliography
%%%%%%%%%%%%%%%%%%%%%%%%%%%%%%%%%%%%%%%%%%%%%%%%%%%%%%%%%%%%%%%%%%%%%

%\bibliographystyle{alpha}
%\bibliography{our}

\end{document}